\documentclass{amsart}
\usepackage{all2023short}
\newcommand{\cO}{\mathcal{O}}
\newcommand{\Gm}{\mathbb{G}_{\mathrm m}}
\newcommand{\gl}{\mathfrak{gl}}

\begin{document}

\title[Generic border subrank]{Border subrank via a
generalised Hilbert-Mumford criterion}

\begin{abstract}
We show that the border subrank of a sufficiently general tensor in
$(\CC^n)^{\otimes d}$ is $\cO(n^{1/(d-1)})$ for $n \to \infty$. Since
this matches the growth rate $\Theta(n^{1/(d-1)})$ for the generic
(non-border) subrank recently established by Derksen-Makam-Zuiddam, we
find that the generic border subrank has the same growth rate. 
In our proof, we use a generalisation of the Hilbert-Mumford criterion
that we believe will be of independent interest.
\end{abstract}

\author{Benjamin Biaggi}
\address{Mathematical Institute, University of Bern,
Alpeneggstrasse 22, 3012 Bern, Switzerland}
\email{benjamin.biaggi@unibe.ch}

\author{Chia-Yu Chang}
\address{Department of Mathematics, Texas A\&M University, 
Mailstop 3368, College Station, TX 77843-3368, USA}
\email{chiayu@tamu.edu}

\author{Jan Draisma}
\address{Mathematical Institute, University of Bern, Sidlerstrasse 5, 3012
Bern, Switzerland; and Department of Mathematics and Computer Science,
P.O. Box 513, 5600 MB, Eindhoven, the Netherlands}
\email{jan.draisma@unibe.ch}

\author{Filip Rupniewski}
\address{Mathematical Institute, University of Bern,
Alpeneggstrasse 22, 3012 Bern, Switzerland}
\email{filip.rupniewski@unibe.ch}

\thanks{BB and FR are supported by Swiss National Science Foundation
(SNSF) project grants 200021-191981 and 200021-227864, and JD is partially supported by
those grants and partially supported by Vici grant 639.033.514 from the
Netherlands Organisation for Scientific Research (NWO). CC is supported by NSF grant AF-2203618}

\maketitle

\section{Introduction and main theorem} \label{sec:Intro}

\subsection*{Introduction}
The subrank, border subrank, and asymptotic subrank play central roles in
several areas including algebraic complexity theory (value of a tensor
\cite{Strassen87}) and 
quantum information theory (asymptotic structure of the capacity of a
quantum channel \cite{Watrous_2018}). 

A central problem in computer science is to study the complexity of
matrix multiplication, which is governed by a constant called the {\it
exponent of matrix multiplication}. This is denoted by $\omega$ and
defined as the infimum over $\tau$ such that the $n\times n$-matrices may
be multiplied using $\cO(n^\tau)$ arithmetic operations. It is known that
$2\leq \omega\leq 3$ and Strassen's original algorithm \cite{Strassen69}
shows that $\omega\leq\log_2(7)<2.81$. Since then, several approaches
have been developed to find an upper bound on $\omega$ or even show that
$\omega$ is 2. The most effective upper bound method found so
far is the {\it laser method}, which was first introduced by Strassen
\cite{Strassen87}. 
A key idea is to find an intermediate tensor $T$ which has low asymptotic
rank (low cost) and is ``close to'' being a matrix multiplication tensor
(high value). The tensor rank measures the ``cost'' of a tensor and the
subrank measures the ``value'' of a tensor. 

It is natural to study how well any sufficiently general tensor $T \in
\CC^n \otimes \CC^n \otimes \CC^n$ performs on these criteria. For rank
and border rank, this is well known: rank and border rank of such a
$T$ are the same and equal to the maximal border rank, namely, $\lceil
\frac{n^3}{3n-2} \rceil$ \cite{Lickteig85} (except for $n=3$, where the
border-rank-$4$ tensors form a hypersurface of degree $9$ defined by
Strassen's equation \cite{Landsberg10,Strassen83}).

Only much more recently has the subrank  of any sufficiently general
tensor been determined: in \cite{Derksen22b}, this is shown to be in the
(small) interval
\begin{equation} \label{eq:Interval} 
\left[3 \lfloor \sqrt{n/3 + 1/4} - 1/2 \rfloor, \lfloor \sqrt{3n-2}
\rfloor\right].
\end{equation} 
The upper bound in this interval was conjectured to be the correct
value for the generic subrank, and indeed, this conjecture was proved in
\cite{Pielasa24}. 

In particular, the generic subrank is $\Theta(n^{1/2})$,
which improves the previous upper bound $\cO(n^{2/3})$ by Bürgisser
\cite{Burgisser1997Degen-6409} and Strassen \cite{STRASSEN+1991+127+180}.

For the border subrank, so far very little was known, other
than that the border subrank of a sufficiently general tensor
in $\CC^n\otimes\CC^n\otimes\CC^n$ is bounded above by $n-1$;
see the invariant-theoretic argument due to Fulvio Gesmundo
in~\cite{Chang22}. The goal of this paper is to dramatically improve
the upper bound on the generic border subrank to $\cO(n^{1/2})$,
matching the growth rate $\Theta(n^{1/2})$ for the generic subrank from
\cite{Derksen22b}. Furthermore, combining our result with the results
from \cite{Derksen22b}, we will establish that the generic subrank and
the generic border subrank do not coincide for $n \gg 0$.

\subsection*{(Border) subrank}

Throughout, we fix an integer $d \geq 2$ and work over an
arbitrary algebraically closed field $K$. Let $V_1,\ldots,V_d$ be
finite-dimensional vector spaces. The {\em subrank} of a tensor
$T \in V_1 \otimes \cdots \otimes V_d$ is the maximal $r$ for which there
exist linear maps $\phi_i:V_i \to K^r,\ i=1,\ldots,d$ such that \[
(\phi_1 \otimes \cdots \otimes \phi_d) T = I_r:=\sum_{i=1}^r e_i^{\otimes
d} \in (K^r)^{\otimes d}. \] Here $e_1,\ldots,e_r$ is the standard
basis of $K^r$, and $I_r$ is called the $r$-th {\em unit tensor}.
For $d=3$, the subrank of $T$ was introduced in theoretical computer
science by Strassen \cite{Strassen87} as a measure of the {\em value}
of $T$---it measures how many independent scalar multiplications can
be linearly embedded in the bilinear map $V_1^* \times V_2^* \to V_3$
encoded by $T$.

The {\em border subrank} of $T$ is the maximal $r$ for which $I_r$
lies in the (Zariski) closure of the set
\[ \{(\phi_1 \otimes \cdots \otimes \phi_d)T \mid 
\phi_i \in \Hom(V_i,K^r), i=1,\ldots,d\}. \]
Clearly, the border subrank of $T$ is at least the subrank of $T$.
When $d=2$, equality holds, and both notions agree with the matrix rank
of $T$. Another immediate observation is that the border subrank of $T$
is at most the rank of $T$ when regarded as a linear map $\bigotimes_{j
\neq i} V_i^* \to V_i$, for any $i \in [d]:=\{1,\ldots,d\}$, and hence
at most $\min_i \dim(V_i)$.

\subsection*{Generic subrank}

The locus of tensors of subrank precisely $r$ in $V_1 \otimes \cdots
\otimes V_d$ is a constructible set, and hence there exists a unique
$r$ for which this locus is dense. This $r$ is called the {\em generic
subrank} of tensors in $V_1 \otimes \cdots \otimes V_d$ and is denoted
$Q(n_1,\ldots,n_d)$ where $n_i=\dim(V_i),\ i=1,\ldots,d$.  It was
proved in \cite{Derksen22b} that $Q(n,n,\ldots,n)=\Theta(n^{1/(d-1)})$
for $n \to \infty$. First, and for general $n_1,\ldots,n_d$, they upper
bound the dimension of the subrank $\geq r$ locus by estimating the
rank of the derivative of a natural morphism parameterising the locus
of tensors of subrank $\geq r$ (this locus is irreducible). This yields
$Q(n,n,\ldots,n) = \cO(n^{1/(d-1)})$. The lower bound involves a clever
construction showing that the rank estimate is essentially tight. For
$d=3$, the results are even sharper and imply that $Q(n,n,n)/\sqrt{3n}
\to 1$ for $n \to \infty$; compare \eqref{eq:Interval}. Very recently,
using the techniques of \cite{Derksen22b} and a very careful analysis of
said derivative, the exact generic subrank of $n_1 \times \cdots \times
n_d$-tensors was determined in \cite{Pielasa24}.

\subsection*{Generic border subrank}

The goal of this paper is to establish an upper bound on the generic
border subrank. We first show that this notion is well-defined. To this
end, let $X_r,X_{\geq r},X_{<r} \subseteq V_1 \otimes \cdots \otimes V_d$
be the loci of tensors of border subrank precisely $r$, at least $r$,
and strictly less than $r$, respectively.

\begin{prop} \label{prop:Constructible}
For any $d,r,V_1,\ldots,V_d$, the set $X_r \subseteq V_1 \otimes \cdots
\otimes V_d$ of tensors of border subrank precisely $r$ is a constructible
set. Therefore, the same holds for the sets $X_{<r}$ and $X_{\geq r}$.
\end{prop}

By the proposition, there is a unique $r$ for which $X_r$ is dense,
and this $X_r$ contains a Zariski open, dense subset of $V_1 \otimes
\cdots \otimes V_d$. This $r$ is called the {\em generic border
subrank} of tensors in $V_1 \otimes \cdots \otimes V_d$.

\begin{thm}[Main Theorem] \label{thm:Main}
The generic border subrank of tensors in $(K^n)^{\otimes d}$
is $\cO(n^{1/(d-1)})$ for $n \to \infty$.
\end{thm}

The Main Theorem follows from the following theorem for general
dimensions.

\begin{thm} \label{thm:Formula}
Let $V_1,\ldots,V_d$ be vector spaces of dimensions
$n_1,\ldots,n_d$, respectively, and let $r$ be a nonnegative integer.
Set $s:=\lfloor r/d \rfloor$. Then the locus $X_{\geq r}$ of tensors of
border subrank $\geq r$ has dimension at most
\[ n_1 \cdots n_d - s^d + 
\sum_{i=1}^d 2 s(n_i-s) + 
r \left(1 + d(r-1) + \sum_{i=1}^d (n_i-r) \right). \]
\end{thm}

Together with the results from \cite{Derksen22b}, the Main Theorem
implies that the growth rates of the generic border subrank and of the
generic subrank are both equal to $\Theta(n^{1/(d-1)})$. However,
their precise values are {\em not} equal: 

\begin{thm} \label{thm:d3}
The generic border subrank of tensors in $(K^n)^{\otimes 3}$ is at least
$\lfloor \sqrt{4n} \rfloor -3 $, and hence, for $n$ sufficiently large,
strictly greater than the generic subrank.
\end{thm}

\subsection*{Proof sketch}

First, we establish Proposition~\ref{prop:Constructible} in
\S\ref{sec:Constructible}. Next, there exist tensors in $V_1 \otimes
\cdots \otimes V_d$ of border subrank $\geq r$ if and only if $\dim(V_i)
\geq r$ for all $i$. Then, choosing linear embeddings $K^r \to V_i$ for
all $i$, we may regard $I_r$ as a tensor in $V_1 \otimes \cdots \otimes
V_d$. The border subrank of a tensor $T$ is $\geq r$ if and only if $I_r$ lies
in the orbit closure of $T$ under the group $G:=\prod_{i=1}^d \GL(V_i)$.

Now it is well known that if a reductive algebraic group $H$ acts on
an affine variety $Z$, $p,q \in Z$ satisfy $q \in \ol{H\cdot p}$,
and the $H$-orbit of $q$ is closed, then there exists an algebraic
group homomorphism ({\em one-parameter subgroup}) $\lambda$ from
the multiplicative group $\Gm$ over $K$ to $H$ such that $\lim_{t \to 0}
\lambda(t) \cdot p \in H \cdot q$---this is the celebrated Hilbert-Mumford criterion;
see, e.g., \cite[Theorem 1.4]{Kempf78}.  Unfortunately, the criterion
does not directly apply in our setting, since the orbit of $I_r$ under
the group $G$ is not closed. Indeed, $\overline{G \cdot I_r}$ is the
set of tensors of ordinary border rank $\leq r$, and this contains many
tensors not in $G \cdot I_r$ and even tensors of ordinary rank $>r$. So
if $T$ has border subrank $r$, it is not clear whether there exists a
one-parameter subgroup $\lambda$ where 
$\lim _{t \to 0} \lambda(t) \cdot T $ is contained in $ G \cdot I_r$.

However, not all is lost: in \S\ref{sec:CIM}, using the
Cartan-Iwahori-Matsumoto decomposition in loop groups, we prove a
generalisation of the Hilbert-Mumford criterion
(Proposition~\ref{Hilbert-Mumford-Prop})
that does apply when $H \cdot q$ is not closed. 
In \S\ref{sec:Proof} we specialise this generalised criterion
to the setting where $Z=V_1 \otimes \cdots \otimes V_d$ and $H=G$ and
$q=I_r$. We show that we can cover the locus of tensors of border subrank
$\geq r$ with countably many constructible subsets, each corresponding to 
a tuple of integer exponents 
of $t$ in a suitable one-parameter subgroup, and for each of
these subsets, we show that its dimension is at most the formula from
Theorem~\ref{thm:Formula}.

\section{Constructibility} \label{sec:Constructible}

We start by showing that the loci $X_{<r},X_r,X_{\geq r}$ of tensors
of border subrank $<r$, $r$, and $\geq r$, respectively, are 
constructible.

\begin{proof}[Proof of Proposition~\ref{prop:Constructible}]
We show that $X_{<r}$ is constructible for all $r$. This implies the
other statements, since $X_{\geq r}$ is the complement of $X_{<r}$ and
$X_r$ is the difference $X_{<r+1} \setminus X_{<r}$. If $r>\dim(V_i)$
for some $i$, then $X_{<r}$ is all of $V_1 \otimes \cdots \otimes V_d$
and we are done. So assume that all $V_i$ have dimension $\geq r$ and
regard $I_r$ as a tensor in $V_1 \otimes \cdots \otimes V_d$. 

By standard results in elimination theory (see, e.g., \cite{Dube90}), there exists an integer
$D$ such that for all $T \in V_1 \otimes \cdots \otimes V_d$ the ideal 
$J_T \subseteq
K[V_1 \otimes \cdots \otimes V_d]$ of polynomials that vanish identically on $\{g T
\mid g \in G\}$ is generated by polynomials of degree $\leq D$. The
border subrank of $T$ is at least $r$ if and only if all polynomials
in the degree-$\leq D$ part $(J_T)_{\leq D}$ of $J_T$ vanish on
$I_r$. Conversely, the border subrank is $<r$ if and only if
\[ \exists f \in K[V^{\otimes d}]_{\leq D}: (f(I_r) \neq 0) \wedge
(\forall g \in G: f(gT)=0). \]
By quantifier elimination (Chevalley's theorem), the locus $X_{<r}$ of
$T$ satisfying the formula above is constructible; here we use that $f$
runs over a finite-dimensional space. 
\end{proof}

Note that, unlike the parameterisation used in \cite{Derksen22b} that implies
that the locus of tensors of subrank $\geq r$ is constructible, the proof
above gives no information about the dimension of that locus.  Such a
bound for border subrank is established in the subsequent sections.

\section{The Cartan-Iwahori-Matsumoto decomposition} \label{sec:CIM}

We recall a classical result from \cite{Iwahori65}. For more along these
lines, we refer to \cite{Alper19} and the references there. Let $K((t))$
denote the field of Laurent series in the variable $t$ with coefficients
in $K$, and $K[[t]]$ its subring of formal power series. For any
(commutative, unital) $K$-algebra $R$ and any affine scheme $X$ over
$K$, the set $X(R)$ denotes the set of $R$-valued points of $X$, i.e.,
the set of $K$-algebra homomorphisms $K[X] \to R$. In particular,
if $G$ is an affine algebraic group over $K$, then $G(K[[t]])$
is a subgroup of the (formal) {\em loop group} $G(K((t)))$. For instance,
$\GL_n(K[[t]])$ is the group of all $n \times n$-matrices with entries
in $K[[t]]$ whose determinant is a unit in $K[[t]]$, i.e., a formal
power series with nonzero constant term. The following theorem follows
from \cite[Corollary 2.17]{Iwahori65}.

\begin{thm} \label{Prop:Cartan-Iwahori-Matsumoto}
Let $G$ be a connected, reductive group over $K$ and let $D$ be a
maximal torus in $G$. Then for any $g=g(t) \in G(K((t)))$ there exist
$h_1(t),h_2(t) \in G(K[[t]])$ and a one-parameter subgroup $\lambda:\Gm
\to D$ such that
\[ g(t)=h_1(t) \lambda(t)  h_2(t)^{-1}. \]
\end{thm}

Here $\lambda(t)$ is regarded as a point in $D(K((t))) \subseteq
G(K((t)))$ as follows: the pull-back of $\lambda$ is an algebra
homomorphism $K[D] \to K[\Gm]=K[t,t^{-1}] \subseteq K((t))$.
For our application to border subrank we only need the special case
of the propositions for (products of) $\GL_n$. In that case, there is the following 
well-known easy proof of this decomposition 
(see \cite[Lemma 7.7]{Mukai03}):

\begin{proof}
Since the matrix entries of $g(t)$ are Laurent series, we can choose an
$a \in \ZZ_{\geq 0}$ so that the entries of $t^a
g(t)$ are formal power series, i.e., elements of $K[[t]]$. Then apply the Smith normal form algorithm
to $t^a g(t)$ to decompose it as above---here we use that $K[[t]]$
is a principal ideal domain all of whose ideals are of the form $(t^b)$
for $b \in \ZZ_{\geq 0}$. Finally, multiply the middle factor by $t^{-a}$
to get the corresponding decomposition for $g(t)$.
\end{proof}

\section{A generalised Hilbert-Mumford criterion}

\begin{prop}
\label{Hilbert-Mumford-Prop}
Let a connected, reductive algebraic group $G$ act on an affine
variety $Z$ and let $p,q \in Z$ such that $q \in \overline{G \cdot p}$.
Then there exists a one-parameter subgroup $\lambda:\Gm \to G$ and a point 
$\tilde{q} \in G \cdot q$, such that 
\[
\lim _{t \rightarrow 0 } \lambda (t) \cdot p =\lim _{t \rightarrow \infty }
\lambda(t) \cdot \tilde{q}.
\]
\end{prop}

In particular, we require that both limits exist!  It is essential for
our application to border subrank that left and right involve 
the {\em same} one-parameter subgroup. The proof that
follows is inspired by the proof of the Hilbert-Mumford criterion in
\cite[Chapter 2, \S1]{Mumford93}, but we are not aware of a previous occurrence of
Proposition~\ref{Hilbert-Mumford-Prop} in the literature. 

\begin{proof}
The variety $Z$ can be embedded as a closed, $G$-stable subvariety of a
finite-dimensional vector space $V$ on which $G$ acts linearly. So we
may assume that $Z=V$ is a representation of $G$.

As $q \in \overline{G \cdot p}$, by standard facts in algebraic
geometry, there exists $g(t) \in G(K ((t)))$ such
that $g(t) \cdot p$ lies in $K[[t]] \otimes V$ and reduces to $q$ when we
set $t$ to zero. We write this as $\lim
_{t\to 0 } g(t) \cdot p= q$. Using Theorem~\ref{Prop:Cartan-Iwahori-Matsumoto}, we can decompose 
\[
g(t) = h_1(t) \mu (t) h_2(t)^{-1}
\]
with $h_1(t),h_2(t) \in G(K[[t]])$ and  one-parameter subgroup
$\mu$. We then have
\[
 h_2(t) \mu(t) h_2(t)^{-1} \cdot p = h_2(t) h_1(t)^{-1} g(t)\cdot  p  \to h_2(0) h_1(0)^{-1} \cdot q=: \tilde{q}  \text{ for } t \to 0, 
 \]
where $\tilde{q}$ is an element in the $G$-orbit of $q$. 

Define the one-parameter subgroup $\lambda (t) = h_2(0) \mu (t) h_2 (0)^{-1}$. 
We will show that $\lim_{t \to 0} \lambda(t)\cdot p = \lim_{t \to
\infty} \lambda(t) \cdot \tilde{q}$. 

There exists a basis $v_1, \ldots , v_n$ of $V$ with $\mu (t) v_i =
t^{a_i} v_i$, where $a_i  \in \ZZ$. The elements $h_2(t) \cdot v_1, \ldots , h_2(t) \cdot v_n $ form a free $K[[t]]$-basis of $K [[t]] \otimes V$ and we can write the vector $p$ as a linear combination in this basis:
\begin{equation}
\label{eq:p_in_powerseriesbasis}
p = \sum _{i=1}^n \xi_i(t) h_2(t) \cdot v_i \text{ for certain } \xi_i(t) \in K[[t]].
\end{equation}
So we have 
\begin{align*}
h_2(t) \mu(t) h_2(t)^{-1} \cdot p  &= h_2(t) \mu (t) \sum _{i=1}^n \xi_i(t)  \cdot v_i  
= h_2(t) \cdot \sum _{i=1}^n \xi_i(t) t^{a_i}  v_i  \\
&= \sum _{i=1}^n  t^{a_i} \xi_i(t) h_2(t) \cdot v_i.
\end{align*}
Since this expression converges for
$t \to 0$, we conclude that if $a_i \leq 0$, then $\xi_i  \in K[[t]]$
is divisible by $t^{-a_i}$. For these $i$, we define $\eta_i(t)  =t^{a_i}
\xi_i(t)  \in K[[t]]$. Now we have
\[ 
 p=\sum_{i:a_i \leq 0} t^{-a_i} \eta_i(t)h_2(t) \cdot v_i + \sum_{i:a_i > 0}
\xi_i(t) h_2(t) \cdot v_i, 
\] 
and the computations above shows that 
\[ 
\tilde{q}=\lim_{t \to 0} h_2(t) \mu(t) h_2(t)^{-1} \cdot p=
\sum_{i:a_i \leq 0} \eta_i(0) h_2 (0) \cdot v_i.
\]
Recalling that $\lambda(t) = h_2(0) \mu (t) h_2(0)^{-1}$, we find
\[
\lambda (t) \cdot \tilde{q} = h_2(0) \cdot \sum_{i:a_i \leq 0} \eta_i(0)
t^{a_i} v_i \to h_2(0) \cdot \sum_{i:a_i = 0} \eta_i(0)  v_i \text{ for } t \to \infty. 
\]
For those $i$ with $a_i =0$, we have $\eta_i(0) = \xi_i (0)$. So 
it remains to show that $\lim _{t \to 0 } \lambda (t) \cdot p
=h_2(0) \cdot \sum_{i:a_i = 0} \xi_i(0)  v_i$. For this, we observe that
setting $t=0$ in \eqref{eq:p_in_powerseriesbasis} yields
\[
p = \sum _{i=1}^n \xi_i(0) h_2(0) \cdot v_i.
\] 
We also saw that $\xi_i (0) = 0$ for $a_i <0$, and hence 
\[
\lambda (t) \cdot p = h_2(0) \cdot \sum _{i:a_i \geq 0} \xi_i t^{a_i} v_i \to h_2(0) \cdot \sum_{i:a_i = 0} \xi_i(0)  v_i \text{ for } t \to 0.
\]
We conclude that $\lim_{t \to 0} \lambda(t) \cdot p = \lim_{t \to \infty}
\lambda(t) \cdot \tilde{q}$, as desired.
\end{proof}

The following example from \cite[Chapter 6.8, Example 1]{Popov94}
shows that the limit $\lim_{t \to 0} \lambda (t) \cdot p$ needs not be contained in 
$G \cdot q$; we construct $\lambda$ and $\tilde{q}$ explicitly. 

\begin{ex}
Suppose $G = \SL _2$ and $V = K[x,y]_3$ is the space of binary cubic forms. An element $g = \big(\begin{smallmatrix}
a & b \\
c & d
\end{smallmatrix}\big)$ acts on $f(x,y) = \sum_{i = 0}^3 a_i x^{3-i}y^i$ by 
\[
g \cdot f(x,y) = f(dx-by,-cx+ay).
\]
Let $p := x^2y$ and $q := x^3$. The orbit $G \cdot q $ is not closed,
as its closure contains $0$. We have $q \in  \overline{G \cdot p}$, as 
\[
\lim _{t \to 0} \begin{pmatrix}
t^{-1} & 0 \\
-t^{-2} & t
\end{pmatrix} \cdot p = \lim_{t \to 0}(x^3 + tx^2y) = q.
\]
Assume that there exists a one-parameter subgroup $\lambda $ such that
$\lim_{t \to 0} \lambda (t) \cdot p $ is contained in $G \cdot q$,
i.e., $G \cdot x^3 \cap \overline{D \cdot x^2y} \neq \emptyset$ for
some one-dimensional torus $D \subset \SL_2$. Then the
stabiliser $G_{x^3}$ contains a conjugate of $D$. This is a contradiction,
as $G_{x^3}$ is a one-dimensional unipotent group.

The Cartan-Iwahori-Matsumoto decomposition of $g$ is
\[
\begin{pmatrix}
t^{-1} & 0 \\
-t^{-2} & t
\end{pmatrix}
= \begin{pmatrix}
t& 1 \\
-1   & 0
\end{pmatrix}
\begin{pmatrix}
t^{-2} & 0 \\
0 & t^{2}
\end{pmatrix}
\begin{pmatrix}
1 & -t^3 \\
0 & 1
\end{pmatrix} =: h_1(t) \mu(t) h_2(t)^{-1}.
\]
We define the one-parameter subgroup $\lambda (t) := h_2(0) \mu (t) h_2(0)^{-1} = \big(\begin{smallmatrix}
  t^{-2} & 0\\
  0 & t^2
\end{smallmatrix}\big)$ and $\tilde{q} := h_2(0)h_1(0)^{-1}q = y^3$.
Now
\[
\lim _{t \to 0 } \lambda (t) \cdot p =\lim _{t \to 0 } t^2x^2y = 0
\] 
and 
\[
\lim _{t \to \infty } \lambda (t) \cdot \tilde{q} =\lim _{t \to \infty
} t^{-6}y^3  =0. \qedhere
\]
\end{ex}

\begin{re}
Note that Proposition~\ref{Hilbert-Mumford-Prop} implies the ordinary
Hilbert-Mumford criterion, as follows. If the orbit of $q$ is closed,
then it follows that $\lim_{t \to \infty} \lambda(t) \cdot \tilde{q}$ is a point in
$G\cdot q$ that is reached as the limit $\lim_{t \to 0} \lambda(t)
\cdot p$. 
\end{re}

\section{Proof of the main theorem} \label{sec:Proof}

We analyse the locus $X_{\geq r}$ of tensors of border subrank $\geq
r$. By Proposition~\ref{prop:Constructible}, this is constructible and
hence has a well-defined dimension. We will cover this locus by countably
many constructible subsets, for each of which we can upper bound the
dimension by some number $N$. These constructible subsets are defined
over $K$, and their $L$-points in fact cover $X_{\geq r}(L)$ for any
field extension $L \supseteq K$. For $L$ uncountable, $X_{\geq r}(L)$
cannot be covered by countably many constructible subsets of dimension
strictly smaller than $\dim(X_{\geq r})$. Hence it follows that $N$
is also an upper bound on the dimension of $X_{\geq r}$.

\begin{re}
The proofs show that one can find bounds that are completely independent
of $K$.
\end{re}

Assume that $T \in V_1 \otimes \cdots \otimes V_d=:V$ has border
subrank $\geq r$. Regarding $I_r$ as an element in $V_1 \otimes
\cdots \otimes V_d$, $I_r$ is in the $G$-orbit closure of $T$. By
Proposition~\ref{Hilbert-Mumford-Prop} applied to $G:=\prod_{i=1}^d
\GL(V_i)$ and $Z:=V$, there exists an element $S \in
G \cdot I_r$ and a one-parameter subgroup $\lambda:\Gm \to G$ such that
\[ \lim_{t \to 0} \lambda(t) \cdot T = \lim_{t \to \infty} \lambda(t)
\cdot S =:S_0, \]
and in particular both limits exist. 

For every $a \in \ZZ$ we define the weight space 
\[ V_a(\lambda) := \{v \in V \mid \forall t \in \Gm: \lambda (t) \cdot v = t^a  v
\}. \] 
We write $V_{>0}(\lambda):=\bigoplus_{a>0} V_a(\lambda)$ and define
$V_{\geq 0}(\lambda),V_{< 0}(\lambda),V_{\leq 0}(\lambda)$ in a
similar manner. We then have $T \in V_{\geq 0}(\lambda), S \in V_{\leq
0}(\lambda)$, and the components $T_0,S_0$ of $T,S$ in $V_0(\lambda)$ are
equal. We will derive an upper bound on $\dim(X_{\geq r})$ by counting
parameters in $K$ needed to determine the components $S_0=T_0$ and
$T_{>0}$.

To formalise this count, fix integers 
\[ 
a_{i1} \leq a_{i2} \leq \cdots \leq a_{in_i},\quad i=1,\ldots,d 
\]
and consider the incidence variety
\[ Y=Y((a_{ij})_{i,j}) :=\{ (\lambda,S,T) \mid S \in G \cdot  I_r \text{ and } \lim_{t \to
\infty}\lambda(t) \cdot S=\lim_{t \to 0}\lambda(t) \cdot T \} \]
where $\lambda=(\lambda_1,\ldots,\lambda_d)$ runs over the one-parameter
subgroups into $G$ such that $\lambda_i:\Gm \to \GL(V_i)$ has weights
$a_{ij},\ j=1,\ldots,n_i$ on $V_i$.

(To be precise, one can take the open subvariety of $\prod_{i=1}^d
\PP(V_i)^{n_i}$ consisting of $d$-tuples in which the $i$-th entry is
a projective basis of $\PP(V_i)$ as the variety parameterising
such $\lambda$, even if this is an over-parameterisation in case
$a_{ij}=a_{il}$ holds for some $i$ and some $j \neq l$.)

By the discussion above, $X_{\geq r}$ is contained in the union, over
all countably many choices of the tuple of integers $a_{ij}$, of the
image of $Y$ under projection onto the third component.

\begin{lm} \label{lm:Tg0}
Set $s:=\lfloor r/d \rfloor$. The number of parameters in $K$ needed to
determine $T_{>0}$, i.e., the dimension of the image of the map 
\[ Y \to V, \quad (\lambda,S,T) \mapsto T_{>0}  \text{ (the component of $T$ in
$V_{>0}(\lambda)$)} \]
 is at most
\[ n_1 \cdots n_d - s^d + \sum_{i=1}^d 2 s (n_i-s). \]
\end{lm}

In the proof we will use the {\em slice rank} of $S$: the
minimal sum $\sum_{i=1}^d \dim(U_i)$ where the $U_i$ are subspaces of
$V_i$ satisfying
\[ S \in \sum_{i=1}^d V_1 \otimes \ldots V_{i-1} \otimes U_i \otimes V_{i+1} \otimes \ldots V_d. \]
For more about the slice rank see \cite{Tao16}. We will use that $I_r$,
and hence $S$, has slice rank $r$; see \cite[Example 5]{Tao16}.

\begin{proof}
Consider $(\lambda,S,T) \in Y$ and let $v_{i1},\ldots,v_{in_i}$ be an
eigenbasis for the $i$-th component $\lambda_i:\Gm \to \GL(V_i)$ of
$\lambda$, so that
\[ \lambda_i(t) \cdot v_{ij}= t^{a_{ij}} v_{ij},\quad t \in \Gm. \]
Now $V_{\leq 0}=V_{\leq 0}(\lambda)$ is the space spanned by all tensors 
\begin{equation} \label{eq:leqzero} v_{1j_1} \otimes \cdots \otimes v_{dj_d} \text{ with }
a_{1j_1} + \cdots + a_{dj_d} \leq 0. \end{equation}
Let $P \subseteq [n_1] \times \cdots \times [n_d]$ be the set of tuples
$(j_1,\ldots,j_d)$ with this property. Since the $a_{ij}$ increase
weakly with
$j$, $P$ is downward closed.

Now if for all $(j_1,\ldots,j_d) \in P$ there exists an $i \in [d]$
with $j_i < s$, then $P$ can be covered by $d
\cdot (s-1)<r$ slices of the form $[n_1] \times \cdots \times [n_{i-1}]
\times \{j \} \times [n_{i+1}] \times \cdots \times [n_d]$ where $j = 1, \ldots ,s-1$, and hence
any linear combination of the tensors  in \eqref{eq:leqzero} has slice
rank $<r$. This contradicts the fact that $S \in V_{\leq 0}$, lying in the
orbit of $I_r$, has slice
rank $r$.  Thus $P$ contains a tuple $(j_1,\ldots,j_d)$ with all $j_i
\geq s$. Since $P$ is downward closed, it then contains the hypercube $[s]^d$.

Let $U_i,W_i \subseteq V_i$ be the spaces spanned by
$v_{i1},\ldots,v_{is}$ and $v_{i,s+1},\ldots,v_{in_i}$, respectively,
so that $V_i=U_i \oplus W_i$.
Then we find that $U_1 \otimes \cdots \otimes U_d \subseteq V_{ \leq
0}$ and hence 
\[ T_{>0} \in V_{>0} \subseteq 
\sum_{i=1}^d V_1 \otimes \cdots \otimes V_{i-1} \otimes W_i \otimes
V_{i+1} \otimes \cdots \otimes V_d.  \]
The right-hand side is a space of dimension $n_1 \cdots n_d - s^d$. 

Furthermore, $U_i$ (respectively, $W_i$) is a point in the Grassmannian
of $s$-dimen\-sio\-nal (respectively, $(n_i-s)$-dimensional) subspaces in
$V_i$. Each of these Grassmannians has dimension $s(n_i-s)$. Adding
these dimensions for $i=1,\ldots,d$ to the upper bound $n_1 \cdots n_d -
s^d$ on the dimension of $V_{>0}$ gives the lemma.
\end{proof}

\begin{lm} \label{lm:S0}
The number of parameters in $K$ needed to determine $S_0=T_0$, i.e.,
the dimension of the image of the map 
\[ Y \to V, \quad  (\lambda,S,T) \mapsto S_0, \]
is at most 
\[ r\left(1+d(r-1)+\sum_{i=1}^d (n_i-r)\right). \]
\end{lm}

\begin{proof}
Consider $(\lambda,S,T) \in Y$, so that $\lim_{t \to \infty}
\lambda(t) \cdot S=S_0$ exists. Define
\[ Q=Q(\lambda):=\{g \in G \mid \lim_{t \to \infty} 
\lambda(t)g \lambda(t)^{-1} \text{ exists in $G$} \}. \]
This is a parabolic subgroup of $G$ (see \cite[page 55]{Mumford93}), and for any $g \in Q$ we have
\[ (g \lambda(t) g^{-1}) \cdot S = 
g \cdot (\lambda(t) g^{-1} \lambda(t)^{-1}) \cdot (\lambda(t) \cdot S) \to (g
g_0) \cdot S_0, \quad (t \to \infty) \]
for some $g_0 \in G$.

Fix a basis $e_{i1},\ldots,e_{in_i}$ of $V_i$ such that 
\[ S=I_r=\sum_{i=1}^r e_{1i} \otimes \cdots \otimes e_{di}, \]
and let $B$ be the Borel subgroup of $G$ consisting of $d$-tuples of upper triangular
matrices relative to these bases. Any two parabolic subgroups
intersect in at least a maximal torus (see, e.g., \cite[Corollary
14.13]{Borel91}), hence some maximal torus
$D$ of $G$ is contained in $Q \cap B$. Any two maximal tori in $Q$
are conjugate (see, e.g., \cite[Corollar 11.3]{Borel91}), and therefore there exists a $g \in Q$ such that $\mu:=g
\lambda g^{-1}$ maps $\Gm$ into $D \subseteq B$. By the previous paragraph,
$S_0':=\lim_{t \to \infty} \mu(t) \cdot S$ lies in the $G$-orbit of $S_0$.

Now let $U_i$ be the subspace of $V_i$ spanned by
$e_{i1},\ldots,e_{ir}$. Since the components of $\mu$ are upper-triangular, the space $U_1
\otimes \cdots \otimes U_d$ is preserved by $\mu$. Hence $S_0'$ also
lies in this space, and it is evidently contained in the orbit closure
of the unit tensor $I_r$ under $\prod_{i=1}^r \GL(U_i)$. 

Since $S_0$ lies in the $G$-orbit of $S_0'$, there also exist
$r$-dimensional subspaces $W_i
\subseteq V_i$ such that $S_0$ lies in $W_1 \otimes \cdots \otimes W_d$
and is contained in the orbit closure of a unit tensor in $W_1 \otimes
\cdots \otimes W_d$. Each $W_i$ is a point in a Grassmannian of dimension
$r(n_i-r)$, and the orbit closure of the unit tensor in $W_1 \otimes
\cdots \otimes W_d$ has dimension $r(1+d(r-1))$: indeed, 
$I_r$ uniquely decomposes as a sum of $r$ tensors in the cone over the
Segre product of the $\PP(W_i)$. Thus $S_0$ sweeps out a variety of
dimension at most 
\[ \sum_{i=1}^d r(n_i-r) + r(1+d(r-1)), \]
as desired. 
\end{proof}

\begin{proof}[Proof of Theorem~\ref{thm:Formula}.]
It suffices to show that the image of the morphism $\pi: Y \to V,\ 
(\lambda,S,T) \mapsto T$ has at most the dimension in
Theorem~\ref{thm:Formula}. This morphism factorises via the map
\[ Y \to V \times V,\quad (\lambda,S,T) \mapsto (S_0,T_{>0}) \]
and the addition map $V \times V \to V$. So an upper bound on the
dimension of $\im(\pi)$ is given by adding the dimensions from
Lemmas~\ref{lm:S0} and~\ref{lm:Tg0}. This yields the upper bound in
the theorem.
\end{proof}

\begin{proof}[Proof of Theorem~\ref{thm:Main}]
Set $n_i=n$ and $r=s \cdot d$ in Theorem~\ref{thm:Formula} (with
$s \in \ZZ_{\geq 0}$). Then 
\begin{align*} 
\dim(X_{\geq r}) &\leq n^d-s^d+2sd(n-s)+r(d(n-r) + 1+d(r-1)) \\
&= n^d-s^d+2sd(n-s)+r(d(n-1)+1)\\
&= n^d-(r/d)^d + r( 2(n-r/d)  + d(n-1) + 1 ).
\end{align*}
Any such multiple $r$ of $d$ for which $X_{\geq r}$ is dense,
i.e., has dimension $n^d$, must therefore satisfy 
\[ (r/d)^d \leq r( 2(n-r/d) + 1 + d(n-1)),  \]
so that 
\[ \frac{r^{d-1}}{d^d} \leq 2(n-r/d) +1 + d(n-1) \leq n(2+d). \]
This shows that $r=\cO(n^{1/(d-1)})$ for $n \to \infty$, as desired.
\end{proof}

\begin{re}
In \cite{Derksen22b}, it is proved that the subrank is not additive
on direct sums of tensors of order $d>2$. The wonderfully simple argument is as
follows: let $\Omega$ be the subset of $V^{\otimes d}$ where
the subrank is generic. By their work, this subrank is $O(n^{1/(d-1)})$,
where $n:=\dim(V)$. Since $\Omega$ is constructible and dense, it contains
a dense open subset of $V^{\otimes d}$. Now let $I_n$ be
a unit tensor in $V^{\otimes d}$ of (the maximal) subrank $n$. Then $I_n-\Omega$
also contains a dense open subset, and hence intersects $\Omega$. So
we find $S,T \in V^{\otimes d}$, both of subrank $O(n^{1/(d-1)})$,
that satisfy $S+T=I_n$. But $S+T$ can be obtained from $S \oplus T \in (V \oplus
V)^{\otimes d}$ by applying the linear map $(\id_V+\id_V)^{\otimes
d}$, so $S \oplus T$ has subrank at least $n$. For $n \gg 0$, this
implies that subrank is not additive on direct sums. The same argument,
now using Proposition~\ref{prop:Constructible} for constructibility
and Theorem~\ref{thm:Main} for the dimension bound, shows that border
subrank is not additive, either.
\end{re}

\section{A lower bound for $d=3$}

\begin{proof}[Proof of Theorem~\ref{thm:d3}]
We fix an integer $r \leq \lfloor \sqrt{4n} \rfloor - 3 $ and show that the
locus $X_{\geq r}$ of tensors in $K^n \otimes K^n \otimes K^n$ is dense. 
To this end, we choose integers $a_{ij},\ i \in [3], j \in [n]$ as
follows: 
\[
a_{ij}:=
\begin{cases} 
2^j & \text{ for $i=1,2$; } \\
-2^{r-j+2} & \text{ for $i=3$ and $j \leq r$; and } \\
0 & \text{ for $i=3$ and $j > r$.}
\end{cases}
\] 
Note that for each $i$, $a_{ij}$ is weakly increasing in $j$.
Moreover, the set 
\[ P:=\{(j,k,l) \mid a_{1j}+a_{2k}+a_{3l} \leq 0 \} \]
equals 
\[ \{(j,k,l) \mid l \leq r \text{ and } j,k \leq r-l+1\}. \]
This is a solid pyramid with its top above one of the corners;
see Figures~\ref{fig:aijk} and~\ref{fig:tensorT}.

\begin{figure}
\includegraphics[width=.7\textwidth]{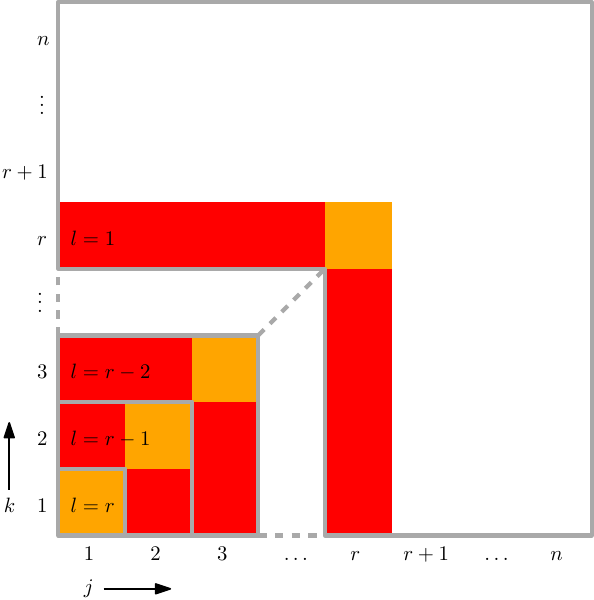}
\caption{The pyramid $P$ seen from above along the $l$-axis.
On the orange corners we have $a_{1j}+a_{2k}+a_{3l}=0$, and on
the red positions (and below these and the orange corners) we have
$a_{1j}+a_{2k}+a_{3l}<0$.}
\label{fig:aijk}
\end{figure}

\begin{figure}
\includegraphics[width=.7\textwidth]{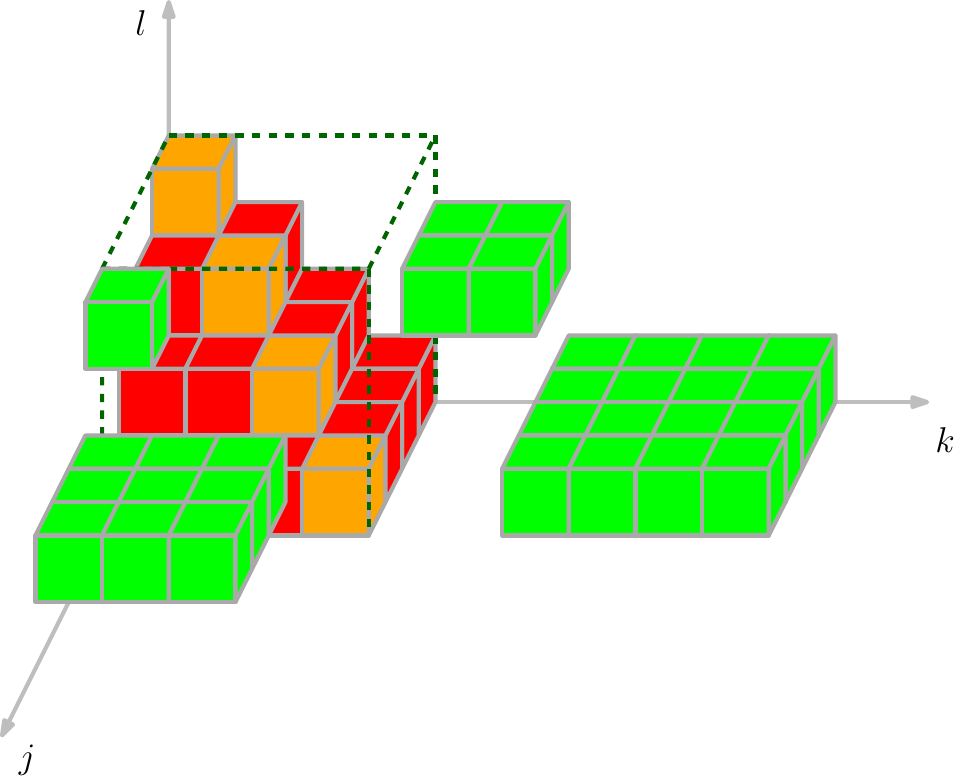}
\caption{The pyramid $P$ in red and orange for $r=4$, and the support of the
additional full-rank matrices in $\tilde{T}$ in green.}
\label{fig:tensorT}
\end{figure}


Now consider the one-parameter subgroup
$\lambda=(\lambda_1,\lambda_2,\lambda_3):\Gm \to \GL_n^3$ where each
$\lambda_i$ has weight vectors $e_1,\ldots,e_n$ with the weights
$a_{i1},\ldots,a_{in}$, respectively; and let $S$ be the tensor that is
zero everywhere except for $1$s on the orange positions. Note that $S$
is a unit tensor of subrank $r$. For any tensor $T$ that agrees
with $S$ on the positions labelled by $P$---a condition we will write
as $T|_P=S|_P$---we have $\lim_{t \to 0} \lambda(t) \cdot T=S$, so that
the border subrank of $T$ is at least $r$. 

Now, much like in \cite{Derksen22b}, we argue that the morphism 
\[ \Psi: \GL_n \times \GL_n \times \{T : T|_P = S|_P\} \to K^n \otimes
K^n \otimes K^n, \quad 
(g_1,g_2,T) \mapsto (g_1, g_2, \id) \cdot T \]
is dominant. This, then, implies that the tensors of border subrank at
least $r$ are dense in $K^n \otimes K^n \otimes K^n$.
To show that $\Psi$ is dominant, we compute the derivative of $\Psi$
at $p:=(\id,\id,\tilde{T})$ for a $\tilde{T}$ to be chosen carefully below. The
tangent space at $\tilde{T}$ to $\{T : T|_P=S|_P\}$ is precisely the
space $V_{>0}$ spanned by the tensors $e_j \otimes e_k \otimes e_l$
with $a_{1j}+a_{2k}+a_{3l}>0$, and $d_p \Psi$ restricted to $V_{>0}$
is the inclusion $V_{>0} \to (K^n)^{\otimes 3}$. So it suffices to show that the restriction of $d_p
\Psi$ to $\gl_n \times \gl_n$ projects surjectively onto
$V_{\leq 0}$, the space spanned by the tensors $e_j \otimes e_k \otimes
e_l$ with $(j,k,l) \in P$. In fact, rather than all of $\gl_n$, we will only use upper triangular
matrices.  Let $E_{ab}$ be the $n \times n$-matrix with zeros everywhere
except for a $1$ on position $(a,b)$. Then $(d_p \Psi) (E_{ab},0)$
is the tensor obtained by putting a copy of the $b$th $j$-slice in the
position where $j=a$ and zeroes elsewhere. We only care about the positions 
in the pyramid $P$, i.e., about $(d_p \Psi)(E_{ab},0)|_P$. 
Similarly in the $k$-direction. 

In the layers with $l=r-s$ with $r > s$ and $s \geq 0$ even, we put any
full-rank $(s+1) \times (s+1)$ matrix $A_s$ in $\tilde{T}$ far enough in front
of $P$, say in positions $[j_s,j_s+s] \times [1,s+1] \times \{l\}$,
so that multiplying this matrix with linear combinations of the matrices
$E_{ab}$ with $a \leq s+1$ and $b \in [j_s,j_s+s]$ yields all possible
matrices in $P$ in layer $l$. We take these such that the intervals
$[j_s,j_s+s]$ are all disjoint. 

Similarly, in the layers with $l=r-s$
with $r>s$ and $s \geq 1$ odd, we put a full-rank $(s+1) \times (s+1)$-matrix
far enough to the right of $P$. See Figure~\ref{fig:tensorT}.
In all positions outside $P$ and outside these matrices,
we choose the entries of $\tilde{T}$ to be $0$.  This ensures that $\Psi$
is dominant, as desired.  

Assume $r$ is even. For the first type of matrices to fit, it suffices
that
\[ n \geq r+(1+3+\cdots+(r-1))=r+r^2/4; \]
and for the second type of matrices to fit, it suffices that 
\[ n \geq r+(2+4+\cdots+r)=r+r(r+2)/4. \]
Assume that $r$ is odd. For the first type of matrices to fit, it
suffices that 
\[ n \geq r+(1+3+\cdots+r)=r+(r+1)^2/4;\]
and for the second type of matrices to fit, it suffices that 
\[ n \geq r+(2+4+\cdots+(r-1))=r+(r-1)(r+1)/4. \]
Summarising, if $n \geq (r+3)^2/4$, then $X_{\geq r}$ is dense. This is
equivalent to 
\[ r \leq \sqrt{4n} -3, \]
and in particular satisfied by $r=\lfloor \sqrt{4n} \rfloor -3$.
Comparing this with the generic subrank in the interval \eqref{eq:Interval}, 
we find that the generic border subrank is strictly greater.
\end{proof}

\section{Further questions}

\begin{enumerate}

\item In \cite{Chang22}, a lower bound on the dimension of the locus
$X_n=X_{\geq n}$ of maximal border subrank tensors is determined for
$d=3$: 
\[ \dim(X_n) \geq (2n^3 + 3n^2 - 2n - 3)/3. \]
Assuming that $n$ is a multiple of $3$, Theorem~\ref{thm:Formula}
yields the following upper bound:
\begin{align*} \dim(X_n) &\leq n^3-(n/3)^3 + 6(n/3)(n-(n/3)) +
n(1+3(n-1))\\
&= \frac{26}{27} n^3 + \frac{13}{3}n^2 - 2n. 
\end{align*}
It would be interesting to find out what is the correct coefficient of
$n^3$ for $n \to \infty$. 

\item In \cite{Derksen22b}, the asymptotic behaviour of the generic
subrank for tensors of order three is determined almost exactly. Can
this be done of the generic border subrank, as well?

\item One can define the {\em Hilbert-Mumford subrank} of $T$ as the
maximal $r$ such that there exists a one-parameter subgroup $\lambda$
into $G$ such that $\lambda(t) \cdot T \to I_r$ for $t \to 0$. This lies
between the subrank of $T$ and the border subrank of $T$. It can be
strictly larger than the subrank of $T$; this follows, for instance,
from \cite{Chang22}: the locus of maximal Hilbert-Mumford subrank
has dimension $\Theta(n^3)$ for $n \to \infty$, whereas the orbit of
$I_n$ has dimension $\Theta(n^2)$.  But can the border subrank of $T$
be strictly larger than the Hilbert-Mumford subrank? 

\end{enumerate}

\bibliographystyle{alpha}
\bibliography{diffeq}

\end{document}